\documentclass[reqno]{amsart}

\usepackage{amsmath}
\usepackage{amsfonts}
\usepackage{amssymb}
\usepackage{amscd}
\usepackage{stmaryrd}
\usepackage[matrix,arrow,cmtip,rotate,curve,arc]{xy}
\SelectTips{cm}{}
\input xy
\xyoption{all}
\usepackage{graphics}
\usepackage{graphicx}
\usepackage{fancyhdr}

\DeclareSymbolFont{bbold}{U}{bbold}{m}{n}
\DeclareMathSymbol{\numberone}{\mathord}{bbold}{`1}

\newcommand{\mrd}{\mathrm d}

\newcommand{\bbR}{\mathbb R}

\newcommand{\mcU}{\mathcal U}
\newcommand{\mcV}{\mathcal V}

\newcommand{\mfX}{\mathfrak X}

\newcommand{\ra}{\rightarrow}

\newcommand{\xlra}[1]{\xrightarrow{\ #1\ }}

\newcommand{\lra}{\longrightarrow}

\newdir{c}{{}*!/-5pt/@^{(}}
\newdir{d}{{}*!/-5pt/@_{(}}
\newdir{s}{{}*!/+10pt/@{}}

\newcommand{\pullback}{\ar@{}[rd]|(.3){\scalebox{1.5}{$\displaystyle\lrcorner$}}}
\newcommand{\pb}[1]{\ar@{}[#1]|(.3){\scalebox{1.5}{$\displaystyle\lrcorner$}}}
\newcommand{\transversepullback}{\ar@{}[rd]|(.3){\scalebox{1.5}{$\displaystyle\lrcorner$}}\ar@{}[rd]|(.25){\pitchfork}}
\newcommand{\tpb}[1]{\ar@{}[#1]|(.3){\scalebox{1.5}{$\displaystyle\lrcorner$}}\ar@{}[rd]|(.25){\pitchfork}}
\newcommand{\pushout}{\ar@{}[lu]|(.3){\scalebox{1.5}{$\displaystyle\ulcorner$}}}

\newcommand{\Charts}{\mathrm{Charts}}
\newcommand{\codim}{\mathop\mathrm{codim}}

\newcommand{\Diff}{\mathrm{Diff}}

\newcommand{\id}{\mathrm{id}}

\newcommand{\imm}{\mathrm{imm}}
\newcommand{\inv}{\mathrm{inv}}
\newcommand{\interior}{\mathop\mathrm{int}}

\newlength{\hlp}
\entrymodifiers={+!!<0pt,\the\fontdimen22\textfont2>}
\theoremstyle{plain}
\newtheorem{theorem}{Theorem}
\newtheorem{lemma}[theorem]{Lemma}
\newtheorem{corollary}[theorem]{Corollary}

\theoremstyle{definition}

\newtheorem*{sconstruction}{Construction}

\theoremstyle{remark}
\newtheorem*{remark}{Remark}

\begin{document}

\title{A generalization of Thom's transversality theorem}
\author{Luk\'a\v{s} Vok\v{r}\'inek}
\address{Department of Mathematics and Statistics\\Masaryk University\\Kotl\'a\v rsk\'a 2\\611 37 Brno\\Czech Republic}
\email{koren@math.muni.cz}
%\date{\today}
\thanks{Research was supported by the grant MSM 0021622409 of the Czech Ministry of Education and by the grant 201/05/2117 of the Grant Agency of the Czech Republic.}
\begin{abstract}
We prove a generalization of Thom's transversality theorem. It gives conditions under which the jet map $f_*|_Y:Y\subseteq J^r(D,M)\ra J^r(D,N)$ is generically (for $f:M\ra N$) transverse to a submanifold $Z\subseteq J^r(D,N)$. We apply this to study transversality properties of a restriction of a fixed map $g:M\ra P$ to the preimage $(j^sf)^{-1}(A)$ of a submanifold $A\subseteq J^s(M,N)$ in terms of transversality properties of the original map $f$. Our main result is that for a reasonable class of submanifolds $A$ and a generic map $f$ the restriction $g|_{(j^sf)^{-1}(A)}$ is also generic. We also present an example of $A$ where the theorem fails.
\end{abstract}

\maketitle

\setcounter{section}{-1}

\section{Introduction}

We start by reminding that for smooth manifolds $M$ and $N$ the set $C^\infty(M,N)$ of smooth maps is endowed with two topologies called weak (compact-open) and strong (Whitney) topology. They agree when $M$ is compact. We say that a subset of a topological space is \emph{residual} if it contains a countable intersection of open dense subsets. The Baire property for $C^\infty(M,N)$ then guarantees that it is automatically dense. This holds for both topologies but is almost exclusively used for the strong one. Clearly every residual subset of $C^\infty(M,N)$ for the strong topology is also residual for the weak topology. The following is our main theorem in which we denote by $J^r_\imm(D,M)$ the subspace of all jets of immersions.

{\renewcommand{\thetheorem}{A}
\begin{theorem} \label{theorem_general}
Let $D$, $M$, $N$ be manifolds, $Y\subseteq J^r_\imm(D,M)$ and $Z\subseteq J^r(D,N)$ submanifolds. Let us further assume that $\sigma_Y\pitchfork\sigma_Z$, where \[\sigma_Y=\sigma|_Y:Y\subseteq J^r(D,M)\rightarrow D\quad\textrm{and}\quad \sigma_Z=\sigma|_Z:Z\subseteq J^r(D,N)\rightarrow D\] are the restrictions of the source maps. For a smooth map $f:M\rightarrow
N$ let $f_*|_Y$ denote the map
\[\xymatrix{
Y \ar@{c->}[r] & J^r_\imm(D,M) \ar[r]^-{f_*} & J^r(D,N)
}.\]
Then the subset
\[\mfX:=\bigl\{f\in C^\infty(M,N)\ \bigl|\ f_*|_Y\pitchfork Z\bigr\}\]
is residual in $C^\infty(M,N)$ with the strong topology, and open if $Z$ is closed (as a subset) and the target map $\tau_Y:Y\rightarrow M$ is proper.
\end{theorem}
\addtocounter{theorem}{-1}}

We are interested in the theorem mainly because of the following two applications, the first of which is the classical theorem of Thom.

{\renewcommand{\thetheorem}{B}
\begin{theorem}[Thom's Transversality Theorem] \label{theorem_thom}
Let $M$, $N$ be manifolds, $Z\subseteq J^r(M,N)$ a submanifold. Then the subset
\[\mfX:=\{f\in C^\infty(M,N)\ |\ j^rf\pitchfork Z\}\]
is residual in $C^\infty(M,N)$. It is moreover open provided that $Z$ is closed (as a subset).
\end{theorem}
\addtocounter{theorem}{-1}}

To explain the second application we need to introduce a bit of notation. Let us consider the following diagram
\[\xymatrix{
& f^{-1}(A) \ar[r] \ar@{c->}[d]^-j & A \ar@{c->}[d] \\
g^{-1}(B) \ar@{c->}[r] \ar[d] & M \ar[r]^{f} \ar[d]^{g} & N \\
B \ar@{c->}[r] & P
}\]
of smooth manifolds and smooth maps between them where $\xymatrix@1{{}\ar@{c->}[r]&{}}$ indicates embeddings. We assume that $f\pitchfork A$ and $g\pitchfork B$ for the two pullbacks to be defined (which we emphasize by saying that they are transverse pullbacks) and also for some technical reasons. Suppose that we fix $g$ and allow ourselves to change $f$ (but only in such a way that $f\pitchfork A$). It is not hard to describe when the composition $gj=g|_{f^{-1}(A)}$ is transverse to $B$. A more general condition on $gj$ is whether it satisfies some form of jet transversality. First we have to solve the problem of not having the source of $J^r(f^{-1}(A),P)$ fixed.

\begin{sconstruction} Let $D$ be a $d$-dimensional manifold and
\[\Diff(\bbR^d,0)=\inv G_0(\bbR^d,\bbR^d)_0\]
the group\footnote{If we wanted to give $\Diff(\bbR^d,0)$ a topology we could do so by inducing the topology via the map $\Diff(\bbR^d,0)\ra J^\infty_0(\bbR^d,\bbR^d)_0$. Or one could replace $\Diff(\bbR^d,0)$ by its image - the subspace of
invertible $\infty$-jets.} of germs at 0 of local diffeomorphisms $\bbR^d\rightarrow\bbR^d$ fixing 0. Define a principal $\Diff(\bbR^d,0)$-bundle
\[\Charts_D=\inv G_0(\bbR^d,D)\xlra{\mathrm{ev}_0} D\]
of germs at 0 of local diffeomorphisms $\bbR^d\rightarrow D$. If $F$ is any manifold with a (smooth in some sense) action of $\Diff(\bbR^d,0)$ then we can construct an associated bundle
\[D[F]:=\Charts_D\times_{\Diff(\bbR^d,0)}F\longrightarrow D.\]
Any bundle of this form is ``local''. Observe that this construction is functorial in $D$ on the category of $d$-dimensional manifolds and local diffeomorphisms. As an example the bundle $J^r(D,P)$ of $r$-jets of maps $D\rightarrow P$ is a local bundle as
\[D[J^r_0(\bbR^d,P)]=\Charts_D\times_{\Diff(\bbR^d,0)}J^r_0(\bbR^d,P)\cong J^r(D,P)\]
where $J^r_0(\bbR^d,P)$ is the subspace of $J^r(\bbR^d,P)$ of $r$-jets with source 0. The bijection is provided by the map
\[[u,\alpha]\mapsto\alpha\circ j^r_{u(0)}(u^{-1}).\]
%It is useful to see how local bundles are constructed. Any chart
%$u:U\rightarrow D$ on $D$ gives a section of $\Charts_D$ over
%$u(U)$:
%\[\sigma_u:u(x)\mapsto u\circ tr_x\]
%where $\tr_x$ denotes the translation by $x$ on $\bbR^d$. If
%$v:V\rightarrow D$ is any other chart (for simplicity of notation
%chosen so that it has the same image) let $\varphi=v^{-1}\circ
%u:U\rightarrow V$ denote the difference map. Then
%\[\sigma_u(u(x))=\sigma_v(u(x))\cdot
%(\tr_{-\varphi(x)}\circ\varphi\circ\tr_x)\] So in practice the
%associated bundle $D[F]$ is a quotient of
%$\coprod\limits_{u:U\rightarrow D}U\times F$ (ranging over all
%charts~$u$) by the equivalence relation coming from the transition
%functions as above. In particular any chart $u:U\rightarrow D$
%gives a trivialization $\tilde{u}:U\times
%J^r_0(\bbR^d,N)\rightarrow J^r(D,N)$ given by
%\[\tilde{u}(x,\alpha)=\alpha\circ j^r_{u(x)}(\tr_{-x}\circ u^{-1})\]

Having a $\Diff(\bbR^d,0)$-invariant submanifold $B\subseteq J^r_0(\bbR^d,P)$ we get an associated subbundle $D[B]\subseteq J^r(D,P)$ for any $d$-dimensional manifold $D$. This allows us to talk about jet transversality conditions on a map $D\rightarrow P$ without specifying what $D$ (and hence also $J^r(D,P)$) is.
\end{sconstruction}

Let $A\subseteq J^s(M,N)$ be a submanifold and $j^sf\pitchfork A$. Then $f^*A:=(j^sf)^{-1}(A)$ is a submanifold of $M$ and for any $\Diff(\bbR^d,0)$-invariant submanifold $B\subseteq J^r_0(\bbR^d,P)$ and the associated submanifold
\[f^*A[B]\subseteq J^r(f^*A,P)\]
we might ask whether $j^r(g|_{f^*A})$ is transverse to $f^*A[B]$. To state the theorem we make the following notation: for a map $g:M\ra P$ we write $g\pitchfork B$ in place of $g_*\pitchfork B$ where again
\[g_*:J^r_{0,\imm}(\bbR^d,M)\lra J^r_0(\bbR^d,P).\]
This condition is satisfied e.g.~by all submersions. Also consider the following map
{\begin{align} \label{eqn_definition_of_kappa}
\kappa:J^{r+s}(M,N)\times_M J^r_{0,\imm}(\bbR^d,M) & \lra J^r_0(\bbR^d,J^s(M,N)), \\
(j^{r+s}_x\varphi,j^r_0\psi) & \longmapsto j^r_0(j^s(\varphi)\circ\psi). \nonumber
\end{align}}%

{\renewcommand{\thetheorem}{C}
\begin{theorem} \label{theorem_main}
Let
\[\xymatrix{
M \ar[r]^{f} \ar[d]^{g} & N \\
P
}\]
be a pair of smooth maps. Let $A\subseteq J^s(M,N)$ and $B\subseteq J^r_0(\bbR^d,P)$ be smooth submanifolds with $d=\dim M-\codim A$. Assuming that $B$ is $\Diff(\bbR^d,0)$-invariant, $\kappa\pitchfork J^r_0(\bbR^d,A)$ and $g\pitchfork B$ the subset
\[\mfX:=\bigl\{f\in C^\infty(M,N)\ \bigl|\ j^sf\pitchfork A,\ j^r(g|_{f^*A})\pitchfork f^*A[B]\bigr\}\]
is residual in $C^\infty(M,N)$ with the strong topology. If either $r=0$ or $s=0$ the transversality condition $\kappa\pitchfork J^r_0(\bbR^d,A)$ is automatically satisfied.
\end{theorem}
\addtocounter{theorem}{-1}}

\begin{remark}
The transversality condition $\kappa\pitchfork J^r_0(\bbR^d,A)$ cannot be removed as we illustrate by an example in the next section. An interesting question arises whether there is a reasonable sufficient condition on $A$ for which the transversality is automatic (an example of such is e.g.~$s=0$).
\end{remark}

\begin{remark}
It is also possible to state conditions under which $\mfX$ is open.
\end{remark}

%\begin{remark}
%Suppose that $A$ is in fact an associated bundle to a $\Diff(\bbR^m,0)$-invariant submanifold of the same name $A\subseteq %J^s_0(\bbR^m,N)$ so that our old $A$ now becomes $M[A]$. Then the transversality condition $\kappa\pitchfork %J^r_0(\bbR^d,M[A])$ reads: the map
%{\begin{align*}
%J^{r+s}_0(\bbR^m,N) & \lra J^r_0(\bbR^d,J^s(\bbR^m,N)) \\
%j^{r+s}_0\psi & \longmapsto j^r_0(j^s(\psi)|_{\bbR^d})
%\end{align*}}%
%is transverse to $J^r_0(\bbR^d,A)$. This can be easily expressed via derivatives using a chart on $N$: suppose that $A$ is %described (as the preimage of 0) by a locally defined submersion
%\[p:J^s_0(\bbR^m,\bbR^n)\ra\bbR^{m-d}\]
%Then the transversality condition is: the collection of maps
%{\begin{align*}
%\partial_Ip:J^{r+s}_0(\bbR^m,\bbR^n) & \lra \bbR^{m-d} \\
%j^{r+s}_0(f) & \longmapsto p(j^s_0(\partial^{|I|}f/\partial x^I))
%\end{align*}}%
%as $I$ is running over all multiindices on $\bbR^d$ with $|I|\leq r$, is independent at the common preimage of 0.
%\end{remark}

\section{An example}
We describe here a family of examples. For $k=1$ (and partially for $k=2$) Theorem~\ref{theorem_main} can be applied whereas for $k\geq 3$ the theorem fails due to $\kappa\not\pitchfork J^r_0(\bbR^d,A)$. We will not be interested in a particular choice of $B$ since the transversality condition in question depends only on $A$. Let us start with the following diagram
\[\xymatrix{
D^k\times M \ar[r]^-f \ar[d] & \bbR \\
D^k
}\]
where we think of $f$ as a family of functions $M\ra\bbR$ parametrized by a disc $D^k$. To fit into our situation we should replace $D^k$ by $\bbR^k$ but the boundary is not the issue in our example. Let $s=1$ and let $A\subseteq J^1(D^k\times M,\bbR)$ denote the subspace of all those jets which have zero derivative in the direction of $M$. Clearly $A$ has codimension $\dim M$ and thus $d=k$.

For a function $\varphi:D^k\times M\ra\bbR$ we have $j^1_{(x,y)}\varphi\in A$ iff the composition
\[\xymatrix{
T_yM \ar@{c->}[r] & T_xD^k\times T_yM \ar[r]^-{\mrd\varphi} & \bbR
}\]
is zero. We express this by saying that the map
\[\mrd|_{TM}:J^1(D^k\times M,\bbR)\ra T^*M\]
describes $A$ as the preimage of 0. We will now show how to describe $J^r_0(\bbR^k,A)$ (or $f^*A$ in fact) in a similar way. Compose the defining equation with an immersion $\psi:\bbR^k\ra D^k\times M$ as in (\ref{eqn_definition_of_kappa}) and differentiate to get
\[\xymatrix@C=25pt{
T_yM\otimes\bbR^k \ar@{c->}[rr]^-{incl\otimes\psi_*} & & S^2(T_xD^k\times T_yM) \ar[r]^-{\mrd^2\varphi} & \bbR
}.\]
Differentiating further and putting all the maps together we get a single map
\[J^{r+1}(D^k\times M,\bbR)\times_M J^r_{0,\imm}(\bbR^k,D^k\times M)\xlra{\chi}T^*M\otimes(\bbR\oplus\bbR^k\oplus\cdots\oplus S^r\bbR^k)^*\]
describing $f^*A$. Moreover $\chi$ is a submersion at $f^*A$ iff $\kappa\pitchfork J^r_0(\bbR^k,A)$. Fixing $j^r_0\psi$ the $T^*M\otimes(S^i\bbR^k)^*$-coordinate is a sum of various restrictions of the derivatives of $\varphi$ with only a single term involving the highest derivative $\mrd^{i+1}\varphi$, namely
\[\xymatrix@C=30pt{
T_yM\otimes S^i\bbR^k \ar@{c->}[rr]^-{incl\otimes S^i\psi_*} & & S^{i+1}(T_xD^k\times T_yM) \ar[r]^-{\mrd^{i+1}\varphi} & \bbR
}.\]
This implies that if the image of the derivative of $\psi_*$ at 0 intersects $TM$ in a subspace of dimension at most 1, the highest order term is a surjective linear map and quite easily the whole map $\chi$ is a submersion (even for a fixed $\psi$). The opposite implication also holds as considerations at pairs with $j^{r+1}\varphi=0$ show.

For $k=1$ we have shown that $\kappa\pitchfork J^r_0(\bbR,A)$ and therefore for a generic map $f$ the fibrewise singularity set $\Sigma_f=f^*A$ is a 1-dimensional submanifold for which the projection $\Sigma_f\ra D^1$ is generic. This means that only regular points and folds appear% (over different points in $D^1$ if one likes)
. The regular points of the projection are exactly the fibrewise Morse singularities of $f$ while folds correspond to the fibrewise $A_3$-singularities (those of the form $x_1^3\pm x_2^2\pm\cdots\pm x_m^2$).

For $k\geq 2$ it is \emph{not} the case that $\kappa\pitchfork J^r_0(\bbR^k,A)$, namely at pairs $(j^{r+s}_x\varphi,j^r_0\psi)$ where the image of the tangential map $\psi_*$ at 0 intersects the tangent space of the fibre $M$ in a subspace of dimension at least 2.

Nevertheless for $k=2$ the conclusion of Theorem~\ref{theorem_main} still holds. This is because the condition $\kappa\pitchfork J^r_0(\bbR^k,A)$ is only required at those $(j^{r+s}_x\varphi,j^r_0\psi)$ for which $\psi(\bbR^d)\subseteq f^*A$ and for a generic map $f$ such $\psi$ cannot be tangent to the fibre $M$ (if that was the case then the rank of the fibrewise Hessian of $f$ at $x$ would drop by 2 and this does not happen generically). Therefore again $\Sigma_f$ is a 2-dimensinal submanifold of $D^2\times M$ partitioned into three parts: the cusps of the projection $\Sigma_f\ra D^2$ (where $f$ attains an isolated fibrewise $A_4$-singularity); the folds form a 1-dimensional submanifold (where the fibrewise $A_3$-singularities occur); and the fibrewise Morse singularities.

For $k=3$ the conclusion of Theorem~\ref{theorem_main} fails. This is due to the fact that for a generic map between 3-manifolds the rank drops at most by 1 whereas for a generic 3-parameter family of functions $M\ra\bbR$ the fibrewise Hessian can drop rank by 2.

\section{Proofs}

First we will show how to translate transversality conditions for a restriction of a fixed map $g:M\ra P$ to the preimage $f^{-1}(A)$ of a submanifold along a map $f:M\ra N$ in terms of transversality conditions on $f$ itself. At the end we prove that such properties are generic in the sense that maps satisfying them form a residual subset of $C^\infty(M,N)$. First version does not involve any jet conditions.

\begin{lemma} \label{lemma_equivalence1}
Let
\[\xymatrix{
& f^{-1}(A) \ar[r] \ar@{c->}[d]^-j & A \ar@{c->}[d] \\
g^{-1}(B) \ar@{c->}[r]^-i \ar[d] & M \ar[r]^{f} \ar[d]^{g} & N \\
B \ar@{c->}[r] & P
}\]
be a diagram where we assume that $f\pitchfork A$ and $g\pitchfork B$. Then the following conditions are equivalent:
\begin{enumerate}
\item[(i)]{$gj\pitchfork B$,}
\item[(ii)]{$fi\pitchfork A$,}
\item[(iii)]{$f^{-1}(A)\pitchfork g^{-1}(B)$.}
%\item[(iv)]{$(f,g)\pitchfork A\times B$ where $(f,g):P\rightarrow M\times N$}
\end{enumerate}
\end{lemma}

\begin{proof} Since (iii) is symmetric it is enough to show the equivalence of (i) and (iii). But (i) is equivalent to the map
\[g_*:T_x f^{-1}(A)\longrightarrow T_{g(x)}P/T_{g(x)}B\]
induced by the derivative of $g$ being surjective for every $x\in f^{-1}(A)\cap g^{-1}(B)$. Because of the assumption $g\pitchfork B$, we have a commutative diagram
\[\xymatrix{
T_x f^{-1}(A) \ar[r] \ar[d] & T_{g(x)}P/T_{g(x)}B \\
T_x M/T_x g^{-1}(B) \ar[ru]_\cong
}\]
and so (i) is equivalent to (iii).
%
%Let us deal with (iv) now (which we actually do not need in the following). It can be phrased as surjectivity of the map %(induced by $(f,g)$)
%\[\varphi:T_xP\longrightarrow (T(M\times N)/T(A\times B))_{(f(x),g(x))}\cong (TM/TA)_{f(x)}\oplus(TN/TB)_{g(x)}\]
%for every $x\in f^{-1}(A)\cap g^{-1}(B)$. Assuming (i) and (ii) the image under $\varphi$ of $T_xf^{-1}(A)$ is $0\oplus(TN/TB)_{g(x)}$ and the image of $T_xg^{-1}(B)$ is $(TM/TA)_{f(x)}\oplus 0$ and so $\varphi$ is surjective. Assuming the surjectivity of $\varphi$ on the other hand we can find $u\in T_xP$ mapping to $(0,v)$ by $\varphi$ for any choice of $v$. Necessarily $u\in T_xf^{-1}(A)$ and so the map
%\[T_x f^{-1}(A)\longrightarrow T_{g(x)}N/T_{g(x)}B\]
%is surjective. This is condition (i).
\end{proof}

Now we will generalize the lemma to jet transversality conditions. Let us recall that for a $\Diff(\bbR^d,0)$-invariant submanifold $B\subseteq J^r_0(\bbR^d,P)$ we have constructed an associated subbundle $D[B]\subseteq J^r(D,P)$ for any $d$-dimensional manifold $D$.

\begin{lemma} \label{lemma_equivalence2}
For $h:D\rightarrow P$ the following conditions are equivalent
\begin{enumerate}
\item[(i)]{$h_*:J^r_{0,\imm}(\bbR^d,D)\rightarrow J^r_0(\bbR^d,P)$ is transverse to $B$,}
\item[(ii)]{$j^r(h):D\rightarrow J^r(D,P)$ is transverse to $D[B]$.}
\end{enumerate}
\end{lemma}

\begin{proof}
Taking associated bundles (i) is clearly equivalent to the transversality of
\[h_*:J^r_\imm(D,D)\rightarrow J^r(D,P)\]
to $D[B]$. Let $j^r_x(k)\in J^r_{x,\imm}(D,D)_y$ be an $r$-jet of a diffeomorphism $k:V\rightarrow W$ between open subsets of $D$. Then we have a diagram
\[\xymatrix{
J^r_{\imm}(V,D) \ar[r]^{h_*} \ar[d]_{\cong}^{k_*} & J^r(V,P) \ar[d]_{\cong}^{k_*} & V[B] \ar@{d->}[l] \ar[d]_{\cong}^{k_*}\\
J^r_{\imm}(W,D) \ar[r]^{h_*} & J^r(W,P) & W[B] \ar@{d->}[l]
}\]
Now $j^r_x(k)$ in the top left corner is mapped by $k_*$ down to $j^r_y(\id)$. Hence we see that it is enough (equivalent) to check the transversality only at $j^r_y(\id)$'s for all $y\in D$ for which $h_*(j^r_y(\id))=j^r_y(h)\in D[B]$. For such $y$ the same diagram shows that every $j^r_x(k)$ with target $y$ is mapped by $h_*$ to $D[B]$. Thus the whole fibre over $y$ of the target map
\[J^r_\imm(D,D)\xlra{\tau} D\]
is mapped to $D[B]$. The target map $\tau$ has a section
\[j^r(\id):D\rightarrow J^r_\imm(D,D)\]
and so (i) is finally equivalent to the composite
\[D\xlra{j^r(\id)}J^r_\imm(D,D)\xlra{h_*}J^r(D,P)\]
being transverse to $D[B]$. This is (ii).
\end{proof}

We say that a map $g:M\rightarrow P$ is transverse to a $\Diff(\bbR^d,0)$-invariant submanifold $B\subseteq J^r_0(\bbR^d,P)$, denoted $g\pitchfork B$, if
\[g_*:J^r_{0,\imm}(\bbR^d,M)\rightarrow J^r_0(\bbR^d,P)\]
%(\textbf{or should one not put an `imm' index to the first term?})
is transverse to $B$. When $r=0$ this is equivalent to the usual transversality of a map to a submanifold. Let $f\pitchfork A$ where $f:M\rightarrow N$ and $A\subseteq N$ is a submanifold. Then we have the following diagram
\[\xymatrix{
J^r_{0,\imm}(\bbR^d,f^{-1}(A)) \ar@{c->}[r] \ar[d]^{j_*} & J^r_0(\bbR^d,f^{-1}(A)) \ar[r] \ar[d]^{j_*} & J^r_0(\bbR^d,A) \ar[d] \\
J^r_{0,\imm}(\bbR^d,M) \ar@{c->}[r] & J^r_0(\bbR^d,M) \ar[r]^-{f_*} & J^r_0(\bbR^d,N)
}\]
where both squares are transverse pullbacks. This can be easily seen in local coordinates. Combining Lemma~\ref{lemma_equivalence1} with Lemma~\ref{lemma_equivalence2} we get:

\begin{lemma} \label{lemma_equivalence3}
Given a diagram
\[\xymatrix{
f^{-1}(A) \ar[r] \ar@{c->}[d]^-j & A \ar@{c->}[d] \\
M \ar[r]^{f} \ar[d]^{g} & N \\
P
}\]
assume that $f\pitchfork A$ and $g\pitchfork B$, where $B\subseteq J^r_0(\bbR^d,P)$ is a $\Diff(\bbR^d,0)$-invariant submanifold with $d=\dim M+\dim A-\dim N$. Then the following conditions are equivalent:
\begin{enumerate}
\item[(i)]{$j^r(gj)\pitchfork (f^{-1}(A))[B]$, where
\[j^r(gj):f^{-1}(A)\rightarrow J^r(f^{-1}(A),P)\]
is the jet prolongation,}
\item[(ii)]{$f_*|_Y\pitchfork J^r_0(\bbR^d,A)$}, where $Y=(g_*)^{-1}(B)$ is defined by a pullback diagram
\[\xymatrix{
Y \ar@{c->}[r] \ar[d] & J^r_{0,\imm}(\bbR^d,M) \ar[d]^{g_*} \\
B \ar@{c->}[r] & J^r_0(\bbR^d,P)
}\]
\end{enumerate}
\end{lemma}

\begin{proof} Applying Lemma~\ref{lemma_equivalence1} to the diagram
\[\xymatrix{
& J^r_{0,\imm}(\bbR^d,f^{-1}(A)) \ar[r] \ar[d]^{j_*} & J^r_0(\bbR^d,A) \ar[d] \\
Y \ar@{c->}[r] \ar[d] & J^r_{0,\imm}(\bbR^d,M) \ar[r]^-{f_*} \ar[d]^{g_*} & J^r_0(\bbR^d,N) \\
B \ar@{c->}[r] & J^r_0(\bbR^d,P)
}\]
gives an equivalence of (ii) with the transversality of
\[(gj)_*:J^r_{0,\imm}(\bbR^d,f^{-1}(A))\longrightarrow J^r_0(\bbR^d,P)\]
to $B$. By Lemma~\ref{lemma_equivalence2} this is equivalent to (i).
\end{proof}

Now that we know how $f$ controls the transversality of a map
defined on the preimage $f^{-1}(A)$ of some submanifold, we would
like to see that this transversality condition (any of the two
equivalent conditions in Lemma~\ref{lemma_equivalence3}) is generic.
This is indeed the case. We first prove a more general result
which at the same time happens to generalize the Thom
Transversality Theorem.

%{\renewcommand{\thetheorem}{A}
%\begin{theorem} \label{theorem_general}
%Let $D$, $M$, $N$ be manifolds, $Y\subseteq J^r_\imm(D,M)$ and $Z\subseteq J^r(D,N)$ submanifolds. Let us further assume that %$\sigma_Y\pitchfork\sigma_Z$, where 
%\[\sigma_Y=\sigma|_Y:Y\subseteq J^r(D,M)\rightarrow D\quad\textrm{and}\quad \sigma_Z=\sigma|_Z:Z\subseteq J^r(D,N)\rightarrow %D\]
%are the restrictions of the source maps. For a smooth map $f:M\rightarrow N$ let $f_*|_Y$ denote the map
%\[Y\subseteq J^r_\imm(D,M)\xlra{f_*} J^r(D,N).\]
%Then the set
%\[\mfX:=\bigl\{f\in C^\infty(M,N)\ \bigl|\ f_*|_Y\pitchfork Z\bigr\}\]
%is residual in $C^\infty(M,N)$ with the strong topology, and open if $Z$ is closed (as a subset) and $\tau_Y:Y\rightarrow M$ %proper.
%\end{theorem}
%\addtocounter{theorem}{-1}}

\begin{proof}[Proof of Theorem~\ref{theorem_general}] This is an application of Lemma~\ref{lemma_general_transversality} from Section~\ref{section_general_transversality}. We have a map
\[\alpha:C^\infty(M,N)\rightarrow C^\infty(Y,J^r(D,N))\]
sending $f$ to $f_*|_Y$. This map is continuous for the weak topology on the target and clearly $\mfX=\{f\in C^\infty(M,N)\ |\ \alpha(f)\pitchfork Z\}$. We have to verify the conditions of Lemma~\ref{lemma_general_transversality}.

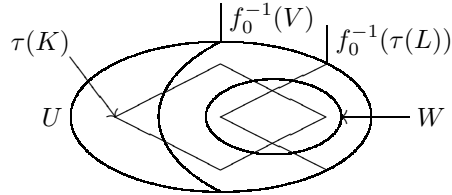
\[\xy 0;/r10mm/:
0,{\ellipse(2,1){-}},
(1.4,0)="a";
(0,.7)**\dir{-};
(-1.4,0)**\dir{-};(-2,.8)*!DR{\tau(K)};c**\dir{-}*\dir{>};
(0,-.7)**\dir{-};"a"**\dir{-},
(0,1);p+(0,.5)**\dir{-},*+!UL!D(.2){f_0^{-1}(V)},(0,-1),{\ellipse`(-2,0){-}},
(\halfroottwo,\halfroottwo)+(\halfroottwo,0);p+(0,.5)**\dir{-}*+!UL!D(.2){f_0^{-1}(\tau(L))},
(0,0)**\dir{-};(\halfroottwo,-\halfroottwo)+(\halfroottwo,0)**\dir{-},
(.7,0),{\ellipse(.9,.5){-}},(1.6,0);(2.5,0)*+!L{W};c**\dir{-}*\dir{>},
(-2,0),*+!R{U},
\endxy\]

Let $f_0\in C^\infty(M,N)$ and $K\subseteq Y$, $L\subseteq Z$ compact discs. We can assume that $\tau(K)$ lies in a coordinate chart $\bbR^m\cong U\subseteq M$ and that $\tau(L)$ lies in a coordinate chart $\bbR^n\cong V\subseteq N$. We use these charts to identify $U$ with $\bbR^m$ and $V$ with $\bbR^n$ when needed. Let $\lambda:U\rightarrow\bbR$ be a compactly supported function such that $\lambda=1$ on a neighborhood of $\tau(K)\cap f_0^{-1}(\tau(L))$ and such that $\lambda=0$ on $U-f_0^{-1}(V)$. This is summarized in the picture above where we put $W=\interior\lambda^{-1}(1)$.

%\textbf{Need a picture!}
We set $Q:=J^r_0(\bbR^m,\bbR^n)$ and identify it both with $J^r_*(U,V)$, where $*$ stands for an arbitrary point in $U$, and also with the space of polynomial mappings $U\ra V$. Then we get a map
\[\beta:Q\rightarrow C^\infty(M,N)\]
sending $q$ to the function $f_0+\lambda q$ where the operations
are interpreted inside $V$ via the chart. It is continuous (in the
strong topology)
%(\textbf{proof?})
and the adjoint map
\begin{equation} \label{eqn_gamma}
\gamma=(\alpha\beta)^\sharp:Q\times Y\rightarrow J^r(D,N)
\end{equation}
is smooth. Thus it is enough to show that (after a suitable
restriction) $\gamma\pitchfork Z$. Clearly $\gamma$ sends
$(q,j^r_x(h))$ to $j^r_x((f_0+\lambda q)h)$. Suppose now that
$h(x)\in W$ so that this equals to
$j^r_x(f_0h+qh)$. By restriction we get a map
\[\delta:Q\cong Q\times\{j^r_x(h)\}\xlra{\gamma}J^r_x(D,V).\]
In the affine structure on $J^r_x(D,V)$ inherited from the chart,
$\delta$ is clearly affine. Identifying $Q$ with $J^r_{h(x)}(U,V)$
the linear part of $\delta$ is just a precomposition with $h$
\begin{equation} \label{eqn_precomposition}
h^*:J^r_{h(x)}(U,V)\rightarrow J^r_{x}(D,V).
\end{equation}
The map $h$, being an immersion, has (locally - near $x$) a left
inverse $\pi$ which then gives a right inverse $\pi^*$ of $h^*$
and so the linear part of $\delta$ is surjective and hence it is a
submersion.

In the horizontal direction our transversality condition $\sigma_Y\pitchfork\sigma_Z$ applies and so $\gamma\pitchfork Z$ on $Q\times \tau_Y^{-1}(W)$. If $f:M\rightarrow N$ is close enough to $f_0$ then
\[\tau(K)\cap f^{-1}(\tau(L))\subseteq W\]
(equivalently $f(\tau(K)-W)\subseteq N-\tau(L)$ which is one of the basic open sets for the compact-open topology) and in particular there is a neighbourhood $Q'$ of $0$ in $Q$ such that $\beta(Q')$ consists only of such maps. Therefore the restriction of $\gamma$ to
\[Q'\times K\lra J^r(D,N)\]
is transverse to $L$ and hence the same is also true in some neighbourhoods of $K$ and $L$ which then form the coverings required in Lemma~\ref{lemma_general_transversality}.

If $\tau_Y$ happens to be proper then $\alpha$ is continuous even in strong topologies and $\mfX$ is a preimage of the open subset of maps $f:Y\rightarrow J^r(D,N)$ transverse to $Z$.
\end{proof}

Now we prove two corollaries of the previous theorem.

%{\renewcommand{\thetheorem}{B}
%\begin{theorem}[Thom's Transversality Theorem] Let $M$, $N$ be
%manifolds, $Z\subseteq J^r(M,N)$ a submanifold. Then the set
%\[\mfX:=\{f\in C^\infty(M,N)\ |\ j^rf\pitchfork Z\}\]
%is residual in $C^\infty(M,N)$. If $Z$ is closed (as a subset)
%then it is also open.
%\end{theorem}
%\addtocounter{theorem}{-1}}

\begin{proof}[Proof of Theorem~\ref{theorem_thom}] We apply Theorem~\ref{theorem_general} to $D=M$ and
%\[\xymatrix{
%\rightbox{Y={}}{M} \ar@{c->}[rr]^{j^r(\id)} \ar[rd]_{\id} & & J^r(M,M)
%\ar[ld]^{\sigma}
%\\
%& M}\]
\[\xymatrix{
Y=M \ar@{c->}[rr]^{j^r(\id)} & & J^r(M,M)
}.\]
As $\sigma_Y=\id=\tau_Y$ it is both proper and transverse
to $\sigma_Z$ for any $Z$.
\end{proof}

\begin{corollary} Let $M$, $N$ be manifolds, $Y\subseteq
J^r_{0,\imm}(\bbR^d,M)$ and $Z\subseteq J^r_0(\bbR^d,N)$
submanifolds. Then the subset
\[\bigl\{f\in C^\infty(M,N)\ \bigl|\ f_*|_Y\pitchfork Z\bigr\}\]
is residual in $C^\infty(M,N)$ with the strong topology, and open
if $Z$ is closed (as a subset) and $\tau_Y:Y\ra M$ proper.
\end{corollary}

\begin{proof} Under the natural identification $\bbR^d\times J^r_0(\bbR^d,N)\cong J^r(\bbR^d,N)$ we can apply Theorem~\ref{theorem_general} to $D=\bbR^d$, the same $M$, $N$ and $Y$ but with $\bbR^d\times Z\subseteq J^r(\bbR^d,N)$ in place of $Z$.
\end{proof}

Now we are finally able to prove our main theorem.

\begin{proof}[Proof of Theorem~\ref{theorem_main}]
Consider first the special case $s=0$ and
\[A=M\times A_0\subseteq M\times N=J^0(M,N).\]
Then $j^0f\pitchfork A$ iff $f\pitchfork A_0$ and under this assumption Lemma~\ref{lemma_equivalence3} provides a translation between $g|_{f^*A}\pitchfork f^*A[B]$ and $f_*|_Y\pitchfork Z$, where $Z=J^r_0(\bbR^d,A_0)\subseteq J^r_0(\bbR^d,N)$. The genericity of $f\pitchfork A_0$ is the usual transversality theorem while the genericity of $f_*|_Y\pitchfork Z$ is the last corollary.

The next step is to replace the transversality to $A_0\subseteq N$ by a jet transversality condition. Let $A\subseteq J^s(M,N)$ be a submanifold and consider
\[\xymatrix{
f^*A \ar[r] \ar@{c->}[d]^-j & A \ar@{c->}[d] \\
M \ar[r]^-{j^sf} \ar[d]^{g} & J^s(M,N) \\
P}\]
It is possible to apply Lemma~\ref{lemma_equivalence3} to this situation and translate $g|_{f^*A}\pitchfork f^*A[B]$ to the transversality of the composition
\[\xymatrix@C=20pt{
Y \ar@{c->}[r] & J^r_{0,\imm}(\bbR^d,M) \ar[rr]^-{(j^sf)_*} & & J^r_0(\bbR^d,J^s(M,N))
}\]
to $J^r_0(\bbR^d,A)$. We cannot however apply Theorem~\ref{theorem_general} directly since we are not interested in all maps $M\ra J^s(M,N)$ but only in the holonomic sections (those of the form $j^sf$). This means that in our proof of Theorem~\ref{theorem_general}, $Q=J^r_*(\bbR^m,J^s(\bbR^m,\bbR^n))$ has to be replaced by its subspace $J^{r+s}_*(\bbR^m,\bbR^n)$ and in general there is no guarantee that the new map $\gamma$ (see (\ref{eqn_gamma})) will be transverse (after restriction) to $Z=J^r_0(\bbR^d,A)$. This is however easily implied by $\kappa\pitchfork J^r_0(\bbR^d,A)$ since again we can arrive at the linear part of $\gamma$ being a composition map (as in~(\ref{eqn_precomposition})) which is then easily identified with $\kappa$.
\end{proof}

%\begin{lemma} Let $f_1,\ldots,f_n:M\rightarrow E$ be smooth
%maps, where $M$ is a smooth manifold, $E$ is some Euclidean space.
%Let $Z\subseteq J^r(M,E)$ be a submanifold. For any
%$f_0:M\rightarrow E$ we can form $f:\bbR^n\times M\rightarrow E$
%by
%\[f(t_1,\ldots,t_n,x)=f_0(x)+t_1f_1(x)+\cdots+t_nf_n(x)\]
%and a partial (or fibrewise) jet prolongation
%$j^r_{\bbR^n}f:\bbR^n\times M\rightarrow J^r(M,E)$
%\[(t_1,\ldots,t_n,x)\mapsto j^r_x f(t_1,\ldots,t_n,-)=j^r_xf_0+t_1j^r_xf_1+\cdots+t_nj^r_xf_n\]
%where on the right we use the vector space structure on $E$ to
%make sense of the sum. Then the set
%\[\left\{f_0\in C^\infty(M,E)\ \bigl|\ j^r_{\bbR^n}f\pitchfork Z\right\}\]
%is residual in the strong topology on $C^\infty(M,E)$.
%\end{lemma}
%
%\begin{proof} One can prove this using the same techniques as in
%the proof of Theorem~\ref{theorem_general}.
%\end{proof}

\section{A general transversality lemma} \label{section_general_transversality}

We will formulate a basic lemma for deciding whether a given family of maps contains a dense subset of maps with a particular transversality property. In a sense this is just the essence of any proof of such a statement. We will be considering maps $\varphi:R\rightarrow C^\infty(S,T)$. We denote by $\varphi^\sharp$ its adjoint
\[\varphi^\sharp:R\times S\rightarrow T.\]

\begin{lemma} \label{lemma_general_transversality}
Let $S$, $T$ be smooth manifolds and $Z\subseteq T$ a submanifold. Let there be given two open coverings: $\mcU$ of $S$ and $\mcV$ of $T$. Let $R$ be a topological space and $\varphi:R\rightarrow C^\infty(S,T)$ a continuous map where $C^\infty(S,T)$ is given the weak topology. Assume that for every $r_0\in R$ and every $U\in\mcU$, $V\in\mcV$ there
is a finite dimensional manifold $Q$ and a continuous map $k:Q\rightarrow R$ with $r_0$ in its image such that
\[Q\times U\xlra{k\times incl}R\times S\xlra{\varphi^\sharp}T\]
is smooth and transverse to $V$. Then the subset
\[\mfX:=\bigl\{r\in R\ \bigl|\ \varphi(r)\pitchfork Z\bigr\}\subseteq R\]
is residual in $R$.
\end{lemma}

\begin{proof}
Following the proof of the Theorem 4.9. of Chapter 4 of \cite{GolubitskyGuillemin}, let us cover $S$ by a countable family
of compact discs $K_i$ that have a neighbourhood $U_i\in\mcU$ and at the same time we choose a covering of $Z$ by a countable family of compact discs $L_j$ that have a neighbourhood $V_j\in\mcV$. Then the set $\mfX$ is a countable intersection of the sets
\[\mfX_{ij}:=\bigl\{r\in R\ \bigl|\ \varphi(r)\pitchfork L_j\textrm{ on }K_i\bigr\}\]
and it is enough to show that each $\mfX_{ij}$ is open and dense. The set $\hat{\mfX}_{ij}$ of maps $S\rightarrow T$ transverse to $L_j$ on $K_i$ is open in $C^\infty(S,T)$ and $\mfX_{ij}=\varphi^{-1}(\hat{\mfX}_{ij})$ so it is also open.

To prove the denseness we fix $r_0\in R$ and choose a map $k:Q\rightarrow R$ with $r_0=k(q_0)$ such that the map
\[l:Q\times U_i\rightarrow T\]
from the statement is smooth and transverse to $V_j$. By the parametric transversality theorem (see e.g.~Theorem 2.7.,
Chapter 3 in \cite{Hirsch}) the points $q\in Q$ for which $l(q,-)\pitchfork V_j$ is dense in $Q$. In particular $q_0$ lies in the closure of this set and hence $r_0$ lies in the closure of its image in $R$. But this image certainly lies in $\mfX_{ij}$.
\end{proof}

\end{document}